\documentclass{amsart}

\usepackage[normalem]{ulem}

\usepackage{lineno}

\usepackage{soul}

\usepackage{amssymb, amsmath, amsthm, hyperref, euscript, color}
\usepackage[dvipsnames]{xcolor}

\usepackage{amscd}

\usepackage{bbm}

\usepackage{enumitem}

\usepackage[all,cmtip]{xy}

\usepackage[english]{babel}

\numberwithin{equation}{section}

\theoremstyle{plain}

\newtheorem{theorem}{Theorem}

\newtheorem{corollary}{Corollary}

\newtheorem{proposition}{Proposition}

\newtheorem{remark}{Remark}

\newtheorem{lemma}{Lemma}

\newtheorem{thm}{Theorem}

\newtheorem{cor}[thm]{Corollary}

\newtheorem{qst}[thm]{Question}

\theoremstyle{definition}

\newtheorem{definition}{Definition}

\newcommand{\Z}{\mathbb{Z}}

\newcommand{\N}{\mathbb{N}}

\newcommand{\C}{\mathbb{C}}

\DeclareMathOperator{\Mod}{mod}
\DeclareMathOperator{\Dim}{dim}

\DeclareMathOperator{\Spin}{Spin}
\DeclareMathOperator{\Kink}{kink}

\DeclareMathOperator{\sign}{sign}

\DeclareMathOperator{\So}{SO}
\DeclareMathOperator{\Sl}{SL}

\begin{document}

\author{Gleb Smirnov and Rafael Torres}

\title[Selection rules for topology change in high-dimensional spacetimes]{Topology change and selection rules for high-dimensional $\Spin(1, n)_0$-Lorentzian cobordisms}

\address{Scuola Internazionale Superiori di Studi Avanzati (SISSA)\\ Via Bonomea 265\\34136\\Trieste\\Italy}

\email{gsmirnov@sissa.it}

\email{rtorres@sissa.it}


\maketitle

\emph{Abstract:} We study necessary and sufficient conditions for the existence of Lorentzian and weak Lorentzian cobordisms between closed smooth manifolds of arbitrary dimension such that the structure group of the frame bundle of the cobordism is $\Spin(1, n)_0$. This extends a result of Gibbons-Hawking on $\Sl(2, \C)$-Lorentzian cobordisms between 3-manifolds and results of Reinhart and Sorkin on the existence of Lorentzian cobordisms. We compute the $\Spin(1, n)_0$-Lorentzian cobordism group for several dimensions. Restrictions on the gravitational kink numbers of $\Spin(1, n)_0$-weak Lorentzian cobordisms are obtained.

\tableofcontents

\section{Introduction}

A cobordism is a triple $(M; N_1, N_2)$ that consists of a smooth compact $(n + 1)$-manifold $M$ with non-empty boundary $\partial M = N_1 \sqcup N_2$, where $N_1$ and $N_2$ are smooth closed $n$-manifolds. Cobordisms have been a central topic of research in topology as studied by Thom \cite{[MilnorStasheff]}, Novikov \cite{[Novikov1]}, Milnor \cite{[Milnor2]}, Quillen \cite{[Quillen]}, Wall \cite{[Wall1]}, and Stong \cite{[Stong]}, among several other outstanding mathematicians. These works include the study of additional topological properties on $(M; N_1, N_2)$ such as the existence of a $\Spin(n +1)$-structure on $M$ that induces a $\Spin(n)$-structure on the boundary manifold; see Section \ref{Section Group Structure}. In this paper, we study the following geometric condition on cobordisms.  

\begin{definition}\label{Definition Lorentzian Cobordism}A Lorentzian cobordism between closed smooth $n$-manifolds $N_1$ and $N_2$ is a pair\begin{equation}((M; N_1, N_2), g)\end{equation} that consists of

(A) a cobordism $(M; N_1, N_2)$,

(B.1) a nonsingular Lorentzian metric $(M, g)$ with a timelike line field $V$,

(C.1) and the boundary $\partial M = N_1 \sqcup N_2$ is spacelike, i.e., $(N_1, g|_{N_1})$ and $(N_2, g|_{N_2})$ are Riemannian manifolds, where $g|_{N_i}$ is the restriction of $g$ to $N_i$. \end{definition}

The study of Lorentzian cobordisms has a long tradition itself and it has been studied by several mathematicians and physicists; see, for example, Reinhart \cite{[Reinhart]}, Yodzis \cite{[Yodzis1], [Yodzis2]}, Sorkin \cite{[Sorkin1]}, Gibbons-Hawking \cite{[GibbonsHawking1], [GibbonsHawking2]}, Chamblin \cite{[Chamblin1]}. The pair of Definition \ref{Definition Lorentzian Cobordism} is a mechanism to study the topology change of spacetime, a classical problem in General Relativity, where the spacetime $(M, g)$ interpolates between a pair of spacelike hypersurfaces $N_1$ and $N_2$ that need not share the same topology. This addresses the possible changes in the topology of the universe as time proceeds. We mention at this point that the spacetimes under study violate causality. Indeed, a result of Geroch says that a compact spacetime whose boundary is either spacelike or timelike must contain at least one closed timelike curve \cite{[Geroch]}; see Remark \ref{Remark Geroch}.

Let us describe the contribution to the subject of the present paper. The main objective of this paper is to establish necessary and sufficient topological conditions that two manifolds $N_1$ and $N_2$ of arbitrary dimension need to meet for there to exist a Lorentzian cobordism $((M; N_1, N_2), g)$ for which the triple $(M; N_1, N_2)$ is a $\Spin$ cobordism. Our main results extends a result of Gibbons-Hawking, which motivated the writing of this paper and that we describe in detail in  Section \ref{Section 3+1}. The situation that they studied on $(3 + 1)$-spacetimes can be summarized as follows. The existence of a timelike vector field as in Item (B.1) implies that the Lorentzian manifold $(M, g)$ in the quadruple of Definition \ref{Definition Lorentzian Cobordism} is time-orientable and whose orientable frame bundle is a principal $\So(1, n)_0$-fiber bundle; see Definition \ref{Definition Spin Structures}. There is a Lorentzian cobordism between any closed orientable 3-manifolds \cite{[Reinhart]} and Milnor has shown that any such 3-manifolds are $\Spin$ cobordant \cite{[Milnor2]}; see Section \ref{Section Lorentzian Cobordisms} and Section \ref{Section Group Structure}. However, Gibbons-Hawking's result revealed that the overlap of Lorentzian and $\Spin$ cobordisms, i.e., where the Lorentzian 4-manifold $(M, g)$ of the pair in Definition \ref{Definition Lorentzian Cobordism} is required to have a principal $\Spin(1, 3)_0$-fiber bundle as its frame bundle, does impose a topological restriction on the boundary 3-manifolds $N_1$ and $N_2$. Our  generalization of Gibbons-Hawking's work provides the necessary and sufficient requirements for $(n + 1)$-spacetimes of arbitrary dimension; see Low \cite{[Low1]} for the case of lower dimensions. The main result in this paper is Theorem \ref{Theorem General}, which is stated and proven in Section \ref{General Result}.

Another contribution of this paper is the study of the algebraic structure of such $\Spin(1, n)_0$-spacetimes as we now briefly motivate, and which are described in more detail in Section \ref{Section Group Structure}. Definition \ref{Definition Lorentzian Cobordism} yields an equivalence relation and partitions manifolds into Lorentzian cobordism classes of $n$-manifolds. The Lorentzian cobordism classes form an abelian group under the operation of disjoint union and a graded ring under the cartesian product of manifolds as established by Reinhart \cite{[Reinhart]}. Classical work of Milnor \cite{[Milnor2], [LawsonMichelsohn]} established the analogous situation for $\Spin$ cobordism classes and he also computed several of the $\Spin$ cobordism groups. We compute several $\Spin(1, n)_0$ cobordism groups and briefly discuss the corresponding graded ring in Section \ref{Section Group Structure} and Section \ref{Section Consequences}. 

More general cobordisms for which the boundary manifolds $N_1$ and $N_2$ are partly spacelike and partly timelike, hence relaxing condition (C.1) in Definition \ref{Definition Lorentzian Cobordism}, have been studied by  Finkelstein-Misner \cite{[FinkelsteinMisner]}, Gibbons-Hawking \cite{[GibbonsHawking2]}, Chamblin-Penrose \cite{[ChamblinPenrose]}, Low \cite{[Low2]} among others. We call these objects weak Lorentzian cobordism; see Section \ref{Section Weak Lorentzian Cobordisms}. Gravitational kinks as introduced by Finkelstein-Misner and studied by Gibbons-Hawking are considered in Section \ref{Section Gravitational Kinks}, where we provide examples of weak Lorentzian cobordisms with prescribed gravitational kinking number. Section \ref{Section GRKinksRestrictions} contains strong restrictions on the gravitational kinking numbers of weak Lorentzian cobordisms with an associated principal $\Spin(1, n)_0$-fiber bundle.

All the manifolds in this paper are $C^{\infty}$, Hausdorf and paracompact.

\section{Preliminaries and background results}

We collect in this section several definitions and fundamental results that are involved in the proofs of our results. We do so, not only with the purpose of making the present paper self-contained, but also hoping it will be useful for a broad audience that ranges from physicists to geometers and topologists. 

\subsection{Lorentzian cobordisms}\label{Section Lorentzian Cobordisms} The purpose of this section is to collect several fundamental results in the study of Lorentzian cobordisms, which are our central object of study. We also include two results that had not previously appeared in the literature. We begin by fleshing out the construction of the Lorentzian metric of Definition \ref{Definition Lorentzian Cobordism} and mention that its existence is equivalent to the existence of line field that is transverse to the boundary. The following definition is due to Reinhart \cite{[Reinhart]}.

\begin{definition}\label{Definition Vector Field Cobordism}A Normal Vector Field cobordism between closed smooth $n$-manifolds $N_1$ and $N_2$ is a pair\begin{equation}((M; N_1, N_2), V),\end{equation} that consists of a cobordism $(M; N_1, N_2)$,

(B.2) a line field $V\in \mathfrak{X}(M)$ such that

(C.2) $V$ is interior normal on $N_1$ and exterior normal on $N_2$. \end{definition}

The line field $V$ is assumed to be orientable and without singularities; see \cite{[Markus]} for further details on line fields.

\begin{lemma}\label{Lemma Equivalent Definitions} Let $\{N_1, N_2\}$ be closed smooth n-manifolds. There exists a Lorentzian cobordism $((M; N_1, N_2)), g)$  if and only if there exits a Normal Vector Field cobordism $((M; N_1, N_2)), V)$. 
\end{lemma}

The following argument is well-known and we include it for the sake of completeness in our exposition (see \cite[Section 2.7]{[HawkingEllis]}, \cite[Section III]{[Yodzis1]}. 

\begin{proof} Suppose there exists a Normal Vector Field Cobordism. In the presence of a line field $V$, one defines a Lorentzian metric by\begin{equation}\label{New LMetric}g(X, Y):= g_R(X, Y) - \frac{2g_R(X, V)g_R(Y, V)}{g_R(V, V)}\end{equation} where $(M, g_R)$ is a Riemannian metric. The line field $V$ is timelike with respect to (\ref{New LMetric}) and $(M, g)$ is a spacetime. Moreover, an orthonormal basis for $g_R$ is an orthonormal basis for (\ref{New LMetric}) too whenever $V$ is one of the basis vectors. We now show the pullback of (\ref{New LMetric}) under embeddings of $N_1$ and $N_2$ yield a Riemannian metric by using the hypothesis of transversality of $V$ at the boundary. Fix a boundary component $N:= N_i$, let $\iota:N\hookrightarrow M$ be the inclusion of the boundary, and denote a basis for $T_pN$ by $\{e_1, \ldots, e_n\}$. Since $V$ is assumed to be transverse to $T_pN$, there exists a vector field $V$ that is orthogonal to every vector in the subspace $\iota_{\ast}(T_p N)\subset T_{\iota(p)} M$ and so that $\{\iota_\ast(e_1), \ldots, \iota_\ast(e_n), V\}$ is a basis for $T_{\iota(p)}M$. The components of (\ref{New LMetric}) under this basis are \[g_{\alpha \beta} =  \left( \begin{array}{cc}
g(V, V) & 0  \\
0 & g(\iota_\ast(e_i), \iota_\ast(e_j)) \end{array}\right)  =  \left( \begin{array}{cc}
g(V, V) & 0  \\
0 & \iota^\ast g(e_i, e_j) \end{array}\right).\] The matrix $g_{\alpha \beta}$ has one negative eigenvalue. Since $g(V, V) < 0$ and $g$ is a Lorentzian metric, we conclude that $\iota^\ast g$ is positive-definite. This implies that the restriction of $g$ to $\partial M$ is a Riemannian metric, and $N_1$ and $N_2$ are spacelike hypersurfaces. 
Suppose there exists a Lorentzian Cobordism and equip the spacetime $(M, g)$ with a Riemannian metric $g_R$. Diagonalize the Lorentzian metric with respect to $g_R$ \cite[\textsection 40]{[Steenrod]}. An eigenvector of negative eigenvalue defines a line field $(\pm V)$. Since $g$ is time orientable, we can resolve the $\pm 1$ ambiguity and the normalize it with respect to $g_R$ to obtain a line field for $M$ as in Definition \ref{Definition Vector Field Cobordism} \cite[p. 40]{[HawkingEllis]}, \cite[Section 2]{[Markus]}. A modification to our previous arguments allows us to conclude that $V$ is transverse to the boundary. 
\end{proof}

Lemma \ref{Lemma Equivalent Definitions} explains the relevance of obtaining an answer to the following question in the study of Lorentzian cobordisms.

\begin{qst}\label{Question 2} Given a cobordism $(M; N_1, N_2)$, under which conditions does there exist a line field $V\in \mathfrak{X}(M)$ as in Definition \ref{Definition Vector Field Cobordism}?

\end{qst}

By using work of Thom amongst others \cite{[Stong], [MilnorStasheff]} to guarantee the existence of a cobordism $(M; N_1, N_2)$, Reinhart provided a complete answer Question \ref{Question 2} in terms of characteristic numbers. A complete answer to Question \ref{Question 2} was also given  independently by Sorkin.

\begin{theorem}\label{Theorem Sorkin} Reinhart \cite[Theorem (1)]{[Reinhart]}, Sorkin \cite[Section II]{[Sorkin1]}. Let $(M; N_1, N_2)$ be a cobordism between closed $n$-manifolds $N_1$ and $N_2$. There exists a line field $V\in \mathfrak{X}(M)$ for which $((M; N_1, N_2), V)$ is a Normal Vector Field cobordism if and only if \begin{equation}\chi(M) = 0\end{equation}whenever the dimension of $M$ is even or\begin{equation}\chi(N_1) = \chi(N_2)\end{equation}provided that the dimension of $M$ is odd.
\end{theorem}

The requirement on the vector field $V$ to be transverse to the boundary $\partial M$ as in Definition \ref{Definition Vector Field Cobordism} is of fundamental importance for the computations in the proof of Theorem \ref{Theorem Sorkin}. Under this requirement and assuming $V$ points outward at every point $p\in \partial M$, the Poincar\'e-Hopf Theorem indicates that the index of the vector field equals the Euler characteristic of the manifold \cite[p. 135]{[Hirsch]}, \cite[$\S$ 6]{[Milnor3]}. We will consider more general vector fields and Lorentzian manifolds whose boundary is not entirely spacelike in the following section.

\subsection{Weak Lorentzian Cobordisms}\label{Section Weak Lorentzian Cobordisms} A smooth manifold admits a Lorentzian metric if and only it has a non-vanishing line field (cf. Proof of Lemma \ref{Lemma Equivalent Definitions}). Every smooth manifold with non-empty boundary has a non-vanishing vector field \cite[2.7 Theorem, Chapter 5.2]{[Hirsch]} (cf. \cite[Proof of Corollary]{[Percell]}), and we can define a Lorentzian metric as in (\ref{New LMetric}) in order to obtain the following proposition. 

\begin{proposition}\label{Proposition Lorentz Boundary} Let $M$ be a smooth manifold with non-empty boundary. There is a Lorentzian metric $(M, g)$ with a timelike vector field $V\in \mathfrak{X}(M)$.

\end{proposition}

A vector field for a spacetime as in Proposition \ref{Proposition Lorentz Boundary}, however, need not be transverse to the boundary $\partial M$. This scenario motivates the following generalization of Definition \ref{Definition Vector Field Cobordism}. 





\begin{definition}\label{Definition Weak Lorentzian Cobordism}\label{Definition weak Lorentzian cobordism}A weak Lorentzian cobordism between closed smooth $n$-manifolds $N_1$ and $N_2$ is a pair\begin{equation}((M; N_1, N_2), g)\end{equation} that consists of 

(A) a cobordism $(M; N_1, N_2)$

(B.4) with a nonsingular Lorentzian metric $(M, g)$ and a timelike non-vanishing vector field $V$,

(C.4) and whose boundary $\partial M = N_1 \sqcup N_2$ is partly spacelike and partly timelike. \end{definition}

The condition of Item (C.4) says that the restriction $g|_{N_i}$ is not positive-definite like in Definition \ref{Definition Lorentzian Cobordism}, instead we allow for there to be points $p_i\in N_i$ where it is degenerate. The interested reader can replicate the scenario that was described in Section \ref{Section Lorentzian Cobordisms} and produced the corresponding versions of Definition \ref{Definition Vector Field Cobordism} and Lemma \ref{Lemma Equivalent Definitions} in this context. We observe that given a cobordism $(M; N_1, N_2)$, Proposition \ref{Proposition Lorentz Boundary} produces a time-orientable Lorentzian metric $g$ for which $((M; N_1, N_2), g)$ is a weak Lorentzian cobordism.

\begin{proposition}\label{Proposition Weak Lorentzian Cobordism}Two closed smooth manifolds are weak Lorentzian cobordant if and only if they are cobordant. 
\end{proposition}

\subsection{Gravitational kink number}\label{Section Gravitational Kinks}We now extend the definition of gravitational kink number due to Gibbons-Hawking \cite{[GibbonsHawking2]} in the case of $(3 + 1)$-spacetimes to arbitrary dimensions; a similar discussion appears in Chamblin \cite{[Chamblin2]}. Let $(M, g)$ be a time-orientable spacetime with non-empty boundary $N = \partial M \subset M$, and consider a Riemannian metric $(M, g_R)$. Diagonalize the Lorentzian metric $(M, g)$ with respect to $g_R$, and obtain a line field $\pm V_g$ from the eigenvector with negative eigenvalue of the diagonalization. The assumption of $(M, g)$ being time-orientable allows us to resolve the sign ambiguity and we normalize $V_g$ to have unit length. Let $S(M)\subset TM$ be the unit sphere bundle with respect to $g_R$. As we consider $M$ to be $(n + 1)$-dimensional, $S(M)$ is a $2n + 1$-manifold. The unit non-vanishing vector field obtained is a section $V_g: M\rightarrow S(M)$. We obtain a bundle over a connected component of the boundary by restricting $S(M)$ to $N$ as in the following diagram
\begin{equation}\label{Diagram of Bundles}
\xymatrix{
S^n\subset \imath^{\ast}(S(M)) \ar[d] \ar[r] &S(M)\ar[d]\\
N \ar@{^{(}->}[r]^{\imath}     &M}
\end{equation}

where $\imath:N\hookrightarrow M$ is the inclusion map and the fiber of the pullback bundle $S(N):= \imath^{\ast}(S(M))\rightarrow N$ is the $n$-sphere. $S(N)$ is a $2n$-manifold with two global sections\begin{equation}\label{Sections}S_{V_g}, S_{\vec n}:N\rightarrow S(N)\end{equation}that are determined by the vector field $V_g$ and the unit inward pointing normal $\vec n$ to the boundary component $N$. The two sections (\ref{Sections}) can be chosen to intersect generically at a finite number of isolated points $p_i$ \cite{[Milnor3]}. Define the $\sign$ of $p_i$ to be $+1$ if the orientation of $S(N)$ coincides with the product of the orientations of the sections (\ref{Sections}), and set $\sign = - 1$ otherwise. Abusing notation and denoting the intersection of the manifolds by $p_i$, we have the following definition. 

\begin{definition}\label{Definition Gravitational Kink Number}The gravitational kink number is defined as (c.f. \cite{[FinkelsteinMisner], [GibbonsHawking2]})\begin{equation}\Kink(N; g):=\sum \sign p_i.\end{equation}

\end{definition}

Gravitational kink numbers were originally introduced by Finkelstein-Misner \cite{[FinkelsteinMisner]} and have been studied by Gibbons-Hawking \cite{[GibbonsHawking2]}, Chamblin-Penrose \cite{[ChamblinPenrose]}, Low \cite{[Low2]} among others. There are several definitions in the literature, and it has been proven that they are all equivalent to Definition \ref{Definition Gravitational Kink Number}; see \cite{[Low2]}. Theorem \ref{Theorem Sorkin} has the following extension to cobordisms as in Definition \ref{Definition Weak Lorentzian Cobordism}. Recall that a smooth $n$-manifold $N$ is stably parallelizable if the Whitney sum of its tangent bundle with a trivial rank one bundle $TN\oplus \epsilon$ is the trivial bundle $\epsilon^{n + 1}$.

\begin{theorem}\label{Theorem Gravitational Kinks}Gibbons-Hawking \cite{[GibbonsHawking2]}, Low \cite[Theorem 3.1]{[Low2]}. Let $((M; N_1, N_2), g)$ be a weak Lorentzian cobordism between closed stably parallelizable $n$-manifolds $N_1$ and $N_2$. We have\begin{equation}\chi(M) = \Kink(\partial M; g)\end{equation}whenever the dimension of $M$ is even or\begin{equation}\frac{1}{2}(\chi(N_2) - \chi(N_1)) = \Kink(\partial M; g)\end{equation}provided that the dimension of $M$ is odd.
\end{theorem}

In the case where the restriction of the Lorentzian metric $g|_{\partial M}$ is positive-definite, the gravitational kink number is $\Kink(\partial M, g) = 0$. Lorentzian cobordisms are characterized by having gravitational kink number zero in the following sense.

\begin{corollary} If $((M; N_1, N_2), g)$ is a Lorentzian cobordism, then\begin{equation}\label{Kink Corollary}\Kink(\partial M, g) = 0.\end{equation}Moreover, if $((M; N_1, N_2), g)$ is a weak Lorentzian cobordism with gravitational kink number as in (\ref{Kink Corollary}), then there exists a Lorentzian metric $(M, g')$ such that $((M; N_1, N_2), g')$ is a Lorentzian cobordism. 
\end{corollary}

We point out the existence of weak Lorentzian cobordisms with prescribed gravitational kinking number in the following result.

\begin{corollary}\label{Corollary Arbitrary Kinking Number} Let $(\hat{M}; N_1, N_2)$ be a cobordism between closed smooth manifolds $N_1$ and $N_2$ of odd dimension and let $t$ be a non-zero integer number. There is a weakly Lorentzian cobordism $((M_t; N_1, N_2), g)$ with\begin{equation}\Kink(\partial M, g) = t.\end{equation}
\end{corollary}

A proof of Corollary \ref{Corollary Arbitrary Kinking Number} follows by constructing connected sums of $\hat{M}$ with other manifolds including projective spaces, and invoking Theorem \ref{Theorem Gravitational Kinks}. The choice of the manifolds to be used in the connected sum depends on the parameter $t$ and need not be unique. Details are left to the interested reader. 

\begin{remark}\label{Remark Geroch} Causality and weak Lorentzian cobordisms. \emph{As it was mentioned in the introduction, a theorem of Geroch says that a non-trivial Lorentzian cobordism must contain at least one closed timelike curve \cite{[Geroch]}. On the other hand, Chamblin-Penrose have shown that weak Lorentzian cobordisms need not contain such curves \cite{[ChamblinPenrose]}; while their argument addresses $(3 + 1)$-spacetimes, it is straight forward to modify it and obtain a generalization to higher dimensional spacetimes}.

\end{remark}

\subsection{$\Spin(1, n)_0$-structures}\label{Section Spin Structures} The structure group of the frame bundle of a time-orientable spacetime is $\So(1, n)_0$ as it was mentioned in the introduction. Our main interest in this paper are Lorentzian cobordisms  whose orientable frame bundle has $\Spin(1, n)_0$ as its structure group. The precise definition is as follows \cite{[Milnor2]}, \cite[Definition 2]{[Chamblin1]}.

\begin{definition}\label{Definition Spin Structures} Let $(M, g)$ be a time-orientable spacetime and let $F(M)$ be its corresponding orientable frame bundle, a principal bundle with structure group $\So(1, n)_0$. The spacetime $(M, g)$ has a $\Spin(1, n)_0$-structure if there is a principal bundle $\bar{F}(M)$ with structure group $\Spin(1, n)_0$ that is a 2:1 covering of $F(M)$ for which the following diagram commutes\begin{equation}
\xymatrix{
\Spin(1, n)_0 \ar[d]\ar[r] &\bar{F}(M)\ar[d]\ar[r] &M\ar[d]^{id}\\
\So(1, n)_0 \ar[r] &F(M) \ar[r] &M.}
\end{equation}

\end{definition} 

The existence of a $\Spin(1, n)_0$-structure on a Lorentzian manifold is a topological property as the following result explains. 

\begin{lemma} A time-orientable spacetime $(M, g)$ admits a $\Spin(1, n)_0$-structure if and only if the tangent bundle of the $(n + 1)$-manifold $M$ admits a $\Spin(n + 1)$-structure.
\end{lemma}

An $(n + 1)$-manifold $M$ admits a $\Spin(n + 1)$-structure if its tangent bundle $TM$ admits a $\Spin(n + 1)$-structure. The second Stiefel-Whitney class is the only obstruction for the tangent bundle of an oriented $(n + 1)$-manifold to admit a $\Spin(n + 1)$-structure \cite[Theorem 2.1]{[LawsonMichelsohn]}. A $\Spin(n + 1)$-structure on an $(n + 1)$-manifold $M$ with non-empty boundary induces a canonical $\Spin(n)$-structure on every boundary component \cite{[LawsonMichelsohn]}. 

\begin{definition}\label{Definition Spin LCobordisms}A $\Spin(1, n)_0$-(weak) Lorentzian cobordism $((M; N_1, N_2), g)$ is a (weak) Lorentzian cobordism with a fiber bundle $\bar{F}(M)$ as in Definition \ref{Definition Spin Structures} that has $\Spin(1, n)_0$ as its structure group. 
\end{definition}

The triple $(M; N_1, N_2)$ in Definition \ref{Definition Spin LCobordisms} is in particular a $\Spin$ cobordism. 

\subsection{$\Spin(1, n)_0$-Lorentzian cobordism group and ring}\label{Section Group Structure}Cobordism is an equivalence relation that partitions the class of closed manifolds into equivalence classes that are called cobordism classes, and which form a group under the operation of disjoint union of manifolds and a graded ring under the cartesian product of manifolds \cite[Chapter I]{[Stong]}. The purpose of this section is to define the $\Spin$-Lorentzian cobordism classes and their algebraic structure. 

We first discuss the $\Spin$ cobordism groups. An $n$-manifold $N$ with a $\Spin(n)$-structure is a $\Spin$ boundary if there exists an $(n + 1)$-manifold $M$ with a $\Spin(n + 1)$-structure and a diffeomorphism $N\rightarrow \partial M$ that induces the $\Spin(n)$-structure on $N$ from the $\Spin(n + 1)$-structure on $M$. In such a case, we say that the manifold $N$ $\Spin$ bounds \cite[p. 90]{[LawsonMichelsohn]}. Two $n$-manifolds $N_1$ and $N_2$ that admit a $\Spin(n)$-structure are equivalent if there exists an $(n + 1)$-manifold $M$ with a $\Spin(n + 1)$-structure such that $\partial M = N_1\sqcup \overline{N_2}$ $\Spin$ bounds \cite[Chapter IV]{[Kirby]}. In such case, we say that $(M; N_1, N_2)$ is a $\Spin$ cobordism. The n-dimensional Spin cobordism group $\Omega_{n}^{\Spin}$ is defined as the set of such equivalence classes equipped with the operation of disjoint union of manifolds \cite[Definition 2.16]{[LawsonMichelsohn]}, \cite[Chapter IX]{[Kirby]}. A manifold that $\Spin$ bounds represents the identity element in $\Omega_n^{\Spin}$. The group $\Omega_n^{\Spin}$ is a commutative group and it has been computed by Milnor for several dimensions \cite{[Milnor2], [LawsonMichelsohn]}, \cite[Chapter IX and Chapter XI]{[Kirby]}. We gather  several of his computations in the following result.

\begin{theorem}\label{Theorem Spin Cobordism Groups} Milnor \cite[p. 201]{[Milnor2]}. The group $\Omega_{n}^{\Spin}$ is zero for $n \in \{3, 5, 6, 7\}$.
\end{theorem}

The cartesian product of manifolds equipped with a $\Spin$ structure has a unique product $\Spin$ structure \cite[Proposition 2.15]{[LawsonMichelsohn]} and multiplication of manifolds yields the $\Spin$ cobordism ring $\Omega^{\Spin}_{\ast} = \underset{n = 0}{\overset{\infty}{\bigoplus}} \Omega_n^{\Spin}$; see \cite{[Milnor2]}, \cite[Chapter II]{[LawsonMichelsohn]}, \cite[Chapter IV]{[Kirby]} for details. 

The scenario with Lorentzian cobordism groups is similar as we now briefly mention. Notice that Lemma \ref{Lemma Equivalent Definitions} allows us to abuse terminology and refer to the cobordism classes of the equivalence relation of Definition \ref{Definition Vector Field Cobordism} and the cobordism classes of the equivalence relation of Definition \ref{Definition Lorentzian Cobordism} as Lorentzian cobordism classes. Reinhart showed that the set of Lorentzian cobordism classes $\mathcal{M}_n$ is a commutative group for $n \in \N$ under the operation of disjoint union of manifolds \cite[Theorem (1)]{[Reinhart]}. The previous discussions motivate the following definition.

\begin{definition}\label{Definition SLorentzian Equivalence Relation} Spin Lorentzian cobordism classes. Two closed $n$-manifolds $N_1$ and $N_2$ that admit a $\Spin(n)$-structure are equivalent if there exists a Lorentzian cobordism $((M; N_1, N_2), g)$ as in Definition \ref{Definition Lorentzian Cobordism} for which $(M; N_1, N_2)$ is a $\Spin$ cobordism. 
\end{definition}

The following result is a canonical extension of Milnor and Reinhart's results. Its proof follows from standard arguments in cobordism theory as given in \cite{[Wall1], [Milnor2], [Reinhart], [Novikov1], [Stong], [Kirby]}.

\begin{theorem}\label{Theorem Group Ring Structure} The set of $\Spin$ Lorentzian cobordism classes equipped with the operation of disjoint union of manifolds\begin{equation}\label{Group S}\Omega_{1, n}^{\Spin_0}\end{equation}is an abelian group. Multiplication of manifolds yields a graded ring\begin{equation}\label{Ring S}\Omega^{\Spin_0}_{1, \ast} = \bigoplus_{n = 0}^\infty \Omega_{1, n}^{\Spin_0}.\end{equation}\end{theorem}

We refer to the group (\ref{Group S}) as the $\Spin(1, n)_0$-Lorentzian cobordism group and the ring (\ref{Ring S}) as the $\Spin(1, n)_0$-Lorentzian cobordism ring.

\section{Existence conditions for $\Spin(1, n)_0$-Lorentzian cobordisms} In this section, we study necessary and sufficient conditions for the existence of a Lorentzian cobordism that can be equipped with a $\Spin(1, n)_0$-structure. In terms of Question \ref{Question 2} and Lemma \ref{Lemma Equivalent Definitions}, we answer the following question.

\begin{qst}\label{Question 3} Given a $\Spin$ cobordism $(\hat{M}; N_1, N_2)$, under which conditions does there exist a $\Spin(1, n)_0$-Lorentzian cobordism $((M; N_1, N_2), g)$?
\end{qst}

A key topological invariant that we will use to answer Question \ref{Question 3} is the following.

\begin{definition}\label{Definition Kervaire}The Kervaire semi-characteristic of a closed smooth orientable $(2q + 1)$-manifold $N$ with respect to a coefficient field $F$ is defined as\begin{equation}\label{SemiChar}\hat{\chi}_{F}(N):= \sum^q_{i = 0}\Dim H^i(N; F) \Mod 2\end{equation} for $q\in \N$. \end{definition}

For the purposes of this paper, we will take $F$ to be either $\Z/2$ or $\mathbb{Q}$.

\subsection{$\Spin(1, 3)_0$-Lorentzian cobordisms: warm-up example}\label{Section 3+1}We now state and prove a stronger version of Gibbons-Hawking's result on $(3 + 1)$-spacetimes that motivated the writing of this note. We compute the ${\Spin(1, 3)_0}$-Lorentzian cobordism group, which does not appear in their paper \cite{[GibbonsHawking1]}. This section serves as a motivation to our main results as well as a prototype to their proofs.

\begin{thm}\label{Theorem GibbonsHawking} Gibbons-Hawking \cite{[GibbonsHawking1], [GibbonsHawking2]}. Let $\{N_1,  N_2\}$ be closed oriented 3-manifolds. The following conditions are equivalent
\begin{enumerate}
\item There exists a $\Spin(1, 3)_0$-Lorentzian cobordism\begin{center}$((M; N_1, N_2), g),$\end{center}where $(M; N_1, N_2)$ is a $\Spin$-cobordism 
\item $\hat{\chi}_{\Z/2}(N_1) = \hat{\chi}_{\Z/2}(N_2)$
\item $M$ is parallelizable.
\end{enumerate}

There is a group isomorphism\begin{equation}\label{4D Isomorphism}\Omega^{\Spin_0}_{1, 3} \rightarrow \Z/2.\end{equation}
\end{thm}

A manifold is parallelizable if its tangent bundle is trivial. Notice that the condition in Item (2) of Theorem \ref{Theorem GibbonsHawking} amounts to $\hat{\chi}_{\Z/2}(\partial M) = 0$. There is an isomorphism $\Spin(1, 3)_0\rightarrow \Sl(2, \C)$ \cite[Theorem 8.4]{[LawsonMichelsohn]}.

\begin{proof} The implication $(1)\Rightarrow (3)$ follows from a theorem of Dold-Whitney \cite{[DoldWhitney]}, which states that every $\So(4)$-bundle is classified by its second Stiefel-Whitney class, and its Euler and Pontrjagin classes. In particular, the tangent bundle $TM$ is trivial if and only if $w_2(M) = e(M) = 0 = p_1(N_i)$. Moreover, a stably parallelizable compact manifold with non-empty boundary is parallelizable. Given that the time-orientable spacetime $(M, g)$ has vanishing Euler characteristic and it is assumed to have a $\Spin(4)$-structure, we conclude that all such characteristic classes are zero. This concludes the proof of the implication $(1)\Rightarrow (3)$. 

The proof of the implication $(3)\Rightarrow (2)$ follows from the proposition

\begin{proposition}\label{Proposition Stably Parallelizable B}\cite{[Kervaire2.1]}, \cite[Proposition 3.4]{[AguilarSeadeVerjovsky]}. Let $M$ be an even-dimensional stably parallelizable manifold with non-empty boundary. Then $\hat{\chi}_{\Z/2}(\partial M) = \chi (M) \mod 2$.
\end{proposition}

We now prove that the implication $(2) \Rightarrow (1)$ holds. The tangent bundle of an orientable 3-manifold is a trivial bundle. In particular, every oriented 3-manifold admits a $\Spin(3)$-structure and it $\Spin$ bounds since the group $\Omega_3^{\Spin}$ is trivial; see Theorem \ref{Theorem Spin Cobordism Groups}. This implies that there exists a $\Spin$ cobordism $(\hat{M}; N_1, N_2)$. Given that $\hat{M}$ admits a $\Spin(4)$-structure, we have\begin{equation}\label{Spin Condition}\hat{\chi}_{\Z/2}(\partial \hat{M}) + \chi(\hat{M}) = 0 \mod 2\end{equation} by a result of Kervaire-Milnor \cite[Lemma 5.9]{[KervaireMilnor]}. It follows from our hypothesis $\hat{\chi}_{\Z/2}(\partial \hat{M}) = \hat{\chi}_{\Z/2}(N_1) + \hat{\chi}_{\Z/2}(N_2) = 0$ and (\ref{Spin Condition}) that the Euler characteristic of $\hat{M}$ is an even integer number. Take non-negative integer numbers $k_1$ and $k_2$ and construct the connected sum\begin{equation}\label{Connected Sum 4D}M:= \hat{M}\#k_1(S^1\times S^3)\#k_2(S^2\times S^2)\end{equation} of $\hat{M}$ with $k_1$ copies of the product of the circle and the 3-sphere $S^1\times S^3$, and $k_2$ copies of the product of two 2-spheres $S^2\times S^2$. Since $\chi(\hat{M})$ is an even number, it is straight forward to see that $k_1$ and $k_2$ can be chosen so that the manifold (\ref{Connected Sum 4D}) has zero Euler characteristic. Theorem \ref{Theorem Sorkin} and Lemma \ref{Lemma Equivalent Definitions} imply that there is a time-orientable Lorentzian metric $(M, g)$ for which $((M; N_1, N_2), g)$ is a $\So(1, 3)_0$-Lorentzian cobordism. The connected sum of two manifolds that admit a $\Spin(n)$-structure can be equipped with a $\Spin(n)$-structure \cite[Remark 2.17]{[LawsonMichelsohn]}. In particular, the manifold (\ref{Connected Sum 4D}) admits a $\Spin(4)$-structure and we conclude that $((M; N_1, N_2), g)$ is a $\Spin(1, 3)_0$-Lorentzian cobordism. This finishes the proof of the implication $(2)\Rightarrow (1)$. Finally, the 3-sphere generates the group $\Omega_{1, 3}^{\Spin_0}$ and the Kervaire semi-characteristic yields the group isomorphism (\ref{4D Isomorphism}).

\end{proof}

Another proof to Gibbons-Hawking's result can be found in \cite[Section 3]{[Chamblin1]}.

\subsection{$\Spin(1, n)_0$-Lorentzian cobordisms}\label{General Result} The complete answer to Question \ref{Question 2} is given in the following theorem.

\begin{thm}\label{Theorem General} Let $\{N_1, N_2\}$ be closed smooth n-manifolds and suppose there exists a $\Spin$ cobordism $(\hat{M}; N_1, N_2)$. 

If $n \neq 7 \mod 8$, then there is a $\Spin(1, n)_0$-Lorentzian cobordism $((M; N_1, N_2), g)$ if and only if 
\begin{enumerate}
\item $\chi(N_1) = \chi(N_2)$ for $n = 0 \mod 2$
\item $\hat{\chi}_{\Z/2}(N_1) = \hat{\chi}_{\Z/2}(N_2)$ for $n = 1, 3, 5 \mod 8$.
\end{enumerate}

If $n = 7 \mod 8$, then there is such a $\Spin(1, n)_0$-Lorentzian cobordism without any further requirements on the manifolds $N_1$ and $N_2$.
\end{thm}

\begin{proof} Item (1) is due to Reinhart; see Theorem \ref{Theorem Sorkin}. We now prove the last claim in the statement of Theorem \ref{Theorem General} and assume that the dimension of $N_1$ and $N_2$ is $n = 7 \mod 8$, i.e., of the form $n = 8q + 8$ for $q$ a non-negative integer number. We start with a given $\Spin$ cobordism $(\hat{M}; N_1, N_2)$, and we need to build a $\Spin$ cobordism $(M; N_1, N_2)$ with $\chi(M) = 0$. Theorem \ref{Theorem Sorkin} will then finish the proof of the claim. With this enterprise in mind, consider the quaternionic projective (2q + 2)-space $\mathbb{H} P^{2q + 2}$. The (real) dimension of $\mathbb{H}P^{2q + 2}$ is $8q + 8$, it admits a $\Spin(8q + 8)$-structure \cite[Example 2.4]{[LawsonMichelsohn]}, and the Euler characteristic is $\chi(\mathbb{H}P^{2q + 2}) = 2q + 3$. There are non-negative integers $k_1$ and $k_2$ such that the connected sum\begin{equation}\label{CS2}M:= \hat{M}\# k_1\mathbb{H}P^{2q + 2} \# k_2 T^{8q + 8}\end{equation} of $\hat{M}$ with $k_1$ copies of the quaternionic projective $(2q + 2)$-space and $k_2$ copies of the $(8q + 8)$-torus, has zero Euler characteristic. Moreover, the manifold (\ref{CS2}) admits a $\Spin(8q + 8)$-structure \cite[Remark 2.17]{[LawsonMichelsohn]} and $\partial M = N_1\sqcup N_2$, i.e., $(M; N_1, N_2)$ is a $\Spin$ cobordism. This concludes the proof of the last claim in the statement of Theorem \ref{Theorem General}.

A proof of the claims for the cases $n = 1, 5 \mod 8$ is obtained by using work of Kevaire-Milnor \cite{[KervaireMilnor]} and Lusztig-Milnor-Peterson \cite{[LMP]} (cf. \cite[p. 240]{[Kervaire2.1]}) as we now explain. Notice that in these cases correspond to dimensions of the form $n = 4q + 1$ for $q\in \N$.

\begin{lemma}\label{Lemma Literature}Kervaire-Milnor \cite{[KervaireMilnor]}, Lusztig-Milnor-Peterson \cite{[LMP]}. Let $M$ be a compact manifold with non-empty boundary that admits a $\Spin(4q + 2)$-structure. The identity\begin{equation}\label{Spin Identity1}\chi(M) + \hat{\chi}_{\Z/2}(\partial M) = 0 \mod 2\end{equation}holds.
\end{lemma}

We proceed to give a proof of Lemma \ref{Lemma Literature}. 

\begin{proof} It is proven in \cite[Lemma 5.6]{[KervaireMilnor]} that the rank of the bilinear pairing\begin{equation}H_{2q + 1}(M; F)\otimes H_{2q + 1}(M; F)\rightarrow F\end{equation} given by the intersection number is congruent  to the sum $\chi(M; F) + \hat{\chi}_F(\partial M)$ modulo 2, where $\chi(M; F)$ is the Euler characteristic with coefficients in the field $F$. With the choice $F = \mathbb{Q}$, we have\begin{equation}\chi(M; \mathbb{Q}) + \hat{\chi}_{\mathbb{Q}}(\partial M) = 0 \mod 2\end{equation} \cite[Proof of Lemma 5.8]{[KervaireMilnor]}. Since $M$ is assumed to admit a $\Spin(4q + 2)$-structure and $\partial M$ has dimension $4q + 1$, we have $\chi(M) + \hat{\chi}_{\Z/2}(\partial M) = 0 \mod 2$ \cite[Theorem]{[LMP]}. This concludes the proof of Lemma \ref{Lemma Literature}.
\end{proof}

Lemma \ref{Lemma Literature} and Theorem \ref{Theorem Sorkin} now yields a proof of Theorem \ref{Theorem General} when the dimension of the manifolds $N_1$ and $N_2$ is of the form $4q + 1$. Suppose that there exists a $\Spin(1, 4q + 1)_0$-Lorentzian cobordism $((M; N_1, N_2), g)$. Theorem \ref{Theorem Sorkin} says $\chi(M) = 0$ and (\ref{Spin Identity1}) implies $\hat{\chi}_{\Z/2}(N_1) = \hat{\chi}_{\Z/2}(N_2)$. Let us prove the converse. The identity (\ref{Spin Identity1}) implies that the Euler characteristic of $\hat{M}$ is even since $\hat{\chi}_{\Z/2}(N_1) =  \hat{\chi}_{\Z/2}(N_2)$ by hypothesis. There are non-negative integer numbers $k_1$ and $k_2$ such that the $(4q + 2)$-manifold\begin{equation}M:= \hat{M}\#k_1(\mathbb{H}P^q\times S^2)\# k_2T^{4q + 2}\end{equation}is a $\Spin$ cobordism $(M; N_1, N_2)$ with $\chi(M) = 0$. Theorem \ref{Theorem Sorkin} implies the existence of a $\Spin(1, 4q + 1)_0$-Lorentzian cobordism $((M; N_1, N_2), g)$. This concludes the proof of the case of manifolds of dimension $4q + 1$ for $q\in \N$. 

Our proof of Item (2) begins with the following triad of lemmas. 

\begin{lemma}\label{Lemma CManifold} Let $M$ be a closed $2q$-manifold that admits a $\Spin(2q)$-structure and suppose $q\neq 0$ $\Mod 4$. The cup product pairing\begin{equation}\cup\colon H^{q}(M) \otimes H^{q}(M) \to H^{2q}(M)\end{equation}is skew-symmetric.
\end{lemma}

\begin{proof}Consider the linear map $H^{q}(M) \rightarrow H^{2q}(M)$ given by $x \mapsto x^2$. We  show that $x^2 = 0$. The cup product $H^{q}(M) \otimes H^{q}(M) \to H^{2q}(M)$ is non-degenerate since $M$ is assumed to be closed. It follows that there exists a unique `'characteristic'' class $v^{q} \in H^{q}(M)$ such that $x^2 = x \cup v$. This class is known as the $q$th Wu class in the literature. A result of  Hopkins-Singer \cite[Lemma E.1]{[HopkinsSinger]} says that $v^q = 0$ provided $q\neq 0 \Mod 4$ and $M$ admits a $\Spin(2q)$-structure. We conclude $x^2 = 0$ as it was claimed.
\end{proof}

\begin{lemma}\label{Lemma wBoundary}Let $M$ be a compact $2q$-manifold that admits a $\Spin(2q)$-structure and with non-empty boundary. Suppose $q\neq 0$ $\Mod 4$. The cup product pairing\begin{equation}\cup\colon H^{q}(M,\partial M) \otimes H^{q}(M,\partial M) \to H^{2q}(M,\partial M)\end{equation} is skew-symmetric.
\end{lemma}

\begin{proof}Set $A:= \partial M$ to ease notation. We show that $x^2 = 0$ for every $x \in H^{q}(M,A)$. According to Kervaire \cite[p. 530] {[Kervaire1]}, the pairing $H^{q}(X,A) \otimes H^{q}(X) \to H^{2q}(X,A)$ is completely orthogonal. A pairing is said to be completely orthogonal if either of the first two groups involved is isomorphic to the group of all homomorphisms of the other into the third. Thus one can conclude that there exists a unique class $s^{q} \in H^{q}(M)$ such that $x^2 = x \cup s$ for every $x \in H^{q}(M,A)$.

Let $P$ be the closed $2q$-manifold that is obtained by gluing two copies of $M$ along $A$ using the identity map to identify the boundaries, i.e., $P:= M\cup_A \overline{M}$, and let $\imath: M \hookrightarrow P$ be the natural inclusion. Since we assumed $M$ to have a $\Spin(2q)$-structure, then $P$ also has a $\Spin(2q)$-structure. Kervaire has shown that\begin{equation}\label{SMC}s^{q} = \imath^{\ast}v^q,\end{equation}where $v^q$ is the $q$th Wu class of $P$ \cite[Lemma (7.3)]{[Kervaire1]}. Lemma \ref{Lemma CManifold} says that $v^q = 0$, which implies $s^q = 0$ by (\ref{SMC}).
\end{proof}

\begin{lemma}\label{Lemma Main}Let $M$ be a compact $2q$-manifold with non-empty boundary, and which admits a $\Spin(2q)$-structure. Suppose $q\neq 0 \Mod 4$. The identity\begin{equation}\chi(M) + \hat{\chi}_{\Z/2}(\partial M) = 0\, \text{\normalfont{mod}}\,2.\end{equation}holds.
\end{lemma}

The following argument is an extension of the proof of the main result in \cite{[GibbonsHawking1]}. Geiges obtained a similar result in the case of orientable 6-manifolds with non-empty boundary \cite[Lemma 8.2.13]{[Geiges]}.

\begin{proof}Set $A:= \partial M$ to ease notation. Consider the exact sequence of homomorphisms of cohomology groups corresponding to an orientable cobordism\begin{equation}\label{Sequence1}0 \to H^{0}(M,A) \to H^{0}(M) \to H^{0}(A) \to \cdots \to H^{q}(M,A) \overset{F}\to H^{q}(M) \to \cdots,\end{equation}where we use $\Z/2$-coefficients. Let $W$ to be the image of $H^q(M,A)$ inside $H^q(M)$ under the group homomorphism that we have labelled $F$ in sequence (\ref{Sequence1}). We apply Lefschetz-Poincar\'e duality between relative cohomology groups and homology groups to  obtain the exact sequence sequence\begin{equation}\label{Sequence2}
0\rightarrow \Z/2 \rightarrow H^0(A) \rightarrow H_{2q - 1}(M)\rightarrow \cdots \rightarrow H^{q - 1}(A)\rightarrow H_q(M) \rightarrow W \rightarrow 0.
\end{equation} Notice that $H^0(M, A)\cong H_n(M) = 0$ since the boundary of $M$ is non-empty. Exactness of the sequences (\ref{Sequence1}) and (\ref{Sequence2}) implies that the alternating sum of the dimensions of these vector spaces over $\Z/2$ vanishes, i.e.,\begin{equation}\label{Expression1}
\sum_{i = 0}^{q} \text{\normalfont{dim}}\,H^{i}(M,A) + 
\sum_{i = 0}^{q - 1} \text{\normalfont{dim}}\,H^{i}(M) +
\sum_{i = 0}^{q - 1} \text{\normalfont{dim}}\,H^{i}(A) + 
\text{\normalfont{dim}}\,W = 0.
\end{equation}

Reversing any of the signs in (\ref{Expression1}) yields the relation
\begin{equation}
\chi(M) +  \hat{\chi}_{\Z/2}(A) = \text{\normalfont{dim}}\,W \mod 2.
\end{equation}

Since the pairing $H^{q}(X,A) \otimes H^{q}(X) \to H^{2q}(X,A)$ is completely orthogonal, the restriction of $H^{q}(X,A) \otimes H^{q}(X,A) \to H^{2q}(X,A)$ to $W$ is non-degenerate. This form is skew-symmetric by Lemma \ref{Lemma wBoundary}. Hence, the dimension of $W$ is even.
This finishes the proof.
\end{proof}

The proof of Item (2) now goes as follows. Suppose there exists a $\Spin(1, n)_0$-Lorentzian cobordism $((M; N_1, N_2), g)$. Theorem \ref{Theorem Sorkin} implies $\chi(M) = 0$, and appealing to Lemma \ref{Lemma Main} we conclude that $\hat{\chi}_{\Z/2}(N_1)  = \hat{\chi}_{\Z/2}(N_1)$. Let us prove the converse. Assume that there exists a $\Spin$ cobordism $(\hat{M}; N_1, N_2)$ and $\hat{\chi}_{\Z/2}(N_1) = \hat{\chi}_{\Z/2}(N_2)$ Lemma \ref{Lemma Main} implies that the Euler characteristic of $\hat{M}$ is even. We now build connected sums that chosen according to the dimension to be considered. For dimensions of the form $n = 8q + 1$, take\begin{equation}\label{Connected Sum3}M:= \hat{M} \# k_1( \mathbb{H}P^{2q}\times S^2) \# k_2 T^{8q + 2}.\end{equation}Since there is a unique closed 1-manifold up to diffeomorphism, the details for this case are left to the reader and we assume $q\in \N$. For dimensions of the form $n = 8q + 3$, we consider\begin{equation}\label{Connected Sum4}M:= \hat{M}  \# k_1(\mathbb{H}P^{2q}\times S^2 \times S^2 )\# k_2 T^{8q + 4}.\end{equation} The case $n = 3$ has been proven in Theorem \ref{Theorem GibbonsHawking} and we will assume $q\in \N$. For $n = 8q + 3$, we take\begin{equation}\label{Connected Sum5}M:= \hat{M} \# k_1(\mathbb{H}P^{2q + 1}\times S^2)\# k_2 T^{8q + 6}\end{equation}where $q$ is a non-negative integer number. The manifolds (\ref{Connected Sum3}), (\ref{Connected Sum4}) and (\ref{Connected Sum5}) yield a $\Spin$ cobordism $(M; N_1, N_2)$ according to the chosen dimension, and in all cases $\chi(M) = 0$. Theorem \ref{Theorem Sorkin} now concludes the proof of the theorem. 
\end{proof}

\subsection{Corollaries and $\Spin(1, n)_0$-Lorentzian cobordism groups}\label{Section Consequences} We collect some consequences of Theorem \ref{Theorem General} and compute several $\Spin(1, n)_0$-Lorentzian cobordism groups using Theorem \ref{Theorem Spin Cobordism Groups}. We start with $(4 + 1)$-spacetimes in the following corollary.

\begin{cor} Let $\{N_1, N_2\}$ be closed smooth 4-manifolds. The following conditions are equivalent.\begin{enumerate}
\item There exists a $\Spin(1, 4)_0$- Lorentzian cobordism $((M; N_1, N_2), g)$.
\item The 4-manifolds $N_1$ and $N_2$ are stably parallelizable and $\chi(N_1) = \chi(N_2)$.
\end{enumerate}

Moreover, there is a group isomorphism\begin{equation}\label{5D Isomorphism}\Omega^{\Spin_0}_{1, 4} \rightarrow \Z\oplus \Z.\end{equation}
\end{cor}

\begin{proof} Let us first show that the implication $(1)\Rightarrow (2)$ holds. Assume there is $\Spin(1, 4)_0$-Lorentzian cobordism $((M; N_1, N_2), g)$. Since the triple $(M; N_1, N_2)$ is a $\Spin$ cobordism, the 4-manifolds $N_1$ and $N_2$ admit a $\Spin(4)$-structure. Thus, the first two Stiefel-Whitney classes $w_1(N_i)$ and $w_2(N_i)$ are both zero for $i = 1, 2$ \cite[Theorem 2.1]{[LawsonMichelsohn]}. A $\Spin$ cobordism is in particular an oriented cobordism, hence the first Pontrjagin classes are $p_1(N_1) = 0 = p_1(N_2)$ \cite{[MilnorStasheff]}. This implies that $N_1$ and $N_2$ are stably parallelizable manifolds by a result of Dold-Whitney \cite{[DoldWhitney]}, which we have already appealed to in the proof of Theorem \ref{Theorem GibbonsHawking}. Theorem \ref{Theorem Sorkin} implies $\chi(N_1) = \chi(N_2)$, and we see that the implication $(1)\Rightarrow (2)$ holds. The proof of the implication $(2)\Rightarrow (1)$ is straight-forward. Since both $N_1$ and $N_2$ are assumed to be stably parallelizable 4-manifolds, their first Pontrjagin classes vanish. This implies that there is a $\Spin$ cobordism $(M; N_1, N_2)$ \cite[Chapter VIII, Theorem 1]{[Kirby]}, and the claim follows from Theorem \ref{Theorem Sorkin}. The group isomorphism (\ref{5D Isomorphism}) is given by the Euler characteristic and the signature of a 4-manifold. The group is generated by the 4-sphere and the K3 surface \cite[Example 2.14]{[LawsonMichelsohn]}.
\end{proof}

Regarding high-dimensional spacetimes, we obtain the following result. 

\begin{cor}\label{Theorem Extension 5D} Let $\{N_1, N_2\}$ be closed smooth $n$-manifolds, which admit a $\Spin(n)$-structure and let $n \in \{5, 6, 7\}$.

$\bullet$ If $n = 5$, there exists a $\Spin(1, 5)_0$-Lorentzian cobordism $((M; N_1, N_2), g)$ if and only if $\hat{\chi}_{\Z/2}(N_1)  = \hat{\chi}_{\Z/2}(N_2)$.

There is a group isomorphism\begin{equation}\label{6D Isomorphism}\Omega^{\Spin_0}_{1, 5} \rightarrow \Z/2.\end{equation}

$\bullet$ If $n = 6$, there exists a $\Spin(1, 6)_0$-Lorentzian cobordism $((M; N_1, N_2), g)$ if and only if $\chi(N_1)  = \chi(N_2)$.

There is a group isomorphism\begin{equation}\label{7D Isomorphism}\Omega^{\Spin_0}_{1, 6} \rightarrow \Z.\end{equation}

$\bullet$ If $n = 7$, there exists a $\Spin(1, 7)_0$-Lorentzian cobordism $((M; N_1, N_2), g)$. 

In particular, $\Omega^{\Spin_0}_{1, 7}$ is the trivial group. 
\end{cor}

There is an isomorphism $\Spin(1, 5)_0\rightarrow \Sl(2, \mathbb{H})$ \cite[Theorem8.4]{[LawsonMichelsohn]}. The Kervaire semi-characteristic yields the group isomorphism (\ref{6D Isomorphism}) and the 5-sphere is a generator for the group. The Euler characteristic yields the group isomorphism (\ref{7D Isomorphism}) and the 6-sphere is a generator for the group. We conclude this section by mentioning the following myriad of examples. 

\begin{cor} Let $N_1$ be a closed smooth $n$-manifold that is a $\Spin$ boundary and suppose $n\geq 5$. For any finitely presented group\begin{equation}\label{Group}G = \langle g_1, \ldots, g_s | r_1, \ldots, r_t\rangle,\end{equation}there is a closed smooth $n$-manifold $N_2(G)$ whose fundamental group is isomorphic to $G$, and a $\Spin(1, n)_0$-Lorentzian cobordism\begin{equation}((M; N_1, N_2(G)), g).\end{equation}

\end{cor}

\begin{proof} Dehn showed that there exists a closed smooth $n$-manifold $N(G, n)$ for every $n\geq 4$ such that the fundamental group is $\pi_1(N(G, n)) \cong G$; see \cite{[Kervaire3]} for a proof due to Kervaire. An inspection of Kervaire's proof  we immediately conclude that $N(G, n)$ $\Spin$-bounds. Fix $n$ and consider the connected sum\begin{equation}\label{Manifold G}N_2(G):=N_2(G, n) \#k_1(S^3\times S^{n - 3})\#k_2(S^2\times S^{n - 2}),\end{equation}where $k_1$ and $k_2$ are non-negative integer numbers. The manifold (\ref{Manifold G}) $\Spin$ bounds by construction. Since we assumed that $N_1$ is a $\Spin$ boundary, it follows that there is a $\Spin$ cobordism $(\hat{M}; N_1, N_2(G))$. If $n\geq 5$, the Seifert-van Kampen theorem indicates that the fundamental group of $N_2(G)$ is isomorphic to $G$. Furthermore, the Euler characteristic of $N_1$ is even since  $N_1$ is the boundary of a compact manifold. When $n$ is even, the non-negative integers $k_1$ and $k_2$ can be chosen so that the equality\begin{equation}\chi(N_1) = \chi(N_2(G, n)\#k_1(S^3\times S^{n - 3})\#k_2(S^2\times S^{n - 2})\end{equation}is satisfied. In the case when $n$ is odd, the integers $p$ and $q$ can be chosen so that\begin{equation}\hat{\chi}_{\Z/2}(N_1) = \hat{\chi}_{\Z/2}(N_2(G, n)\#k_1(S^3\times S^{n - 3})\#k_2(S^2\times S^{n - 2}).\end{equation}The corollary now follows from Theorem \ref{Theorem Sorkin}. 
\end{proof}

\section{Restriction on gravitational kink numbers of $\Spin(1, n)_0$-weak Lorentzian Cobordisms}\label{Section GRKinksRestrictions}

Proposition \ref{Proposition Weak Lorentzian Cobordism} and Theorem \ref{Theorem Spin Cobordism Groups} immediately imply the following result. 

\begin{cor}Let $n\in \{3, 5, 6, 7\}$. There is a $\Spin(1, n)_0$-weak Lorentzian cobordism $((M; N_1, N_2), g)$ for every pair of closed smooth $n$-manifolds $N_1$ and $N_2$ that admit a $\Spin(n)$-structure.
\end{cor}

Corollary \ref{Corollary Arbitrary Kinking Number} indicates the existence of a $\So(1, 2q + 1)_0$-weak Lorentzian cobordism with prescribed gravitational kinking number between any given cobordant manifolds of odd dimension. If the weak Lorentzian cobordism is required to carry a $\Spin$ structure, we observe that the parity of the gravitational kinking number and the Euler semi-characteristic of the boundary must coincide in the following sense. 
 
\begin{proposition}\label{Proposition Restriction Kinking Spin}If $((M; N_1, N_2), g)$ is a $\Spin(1, 2q + 1)_0$-weak Lorentzian cobordism, then\begin{equation}\hat{\chi}_{\Z/2}(\partial M) = \Kink(\partial M; g) \Mod 2.\end{equation}
\end{proposition}

The proof of Proposition \ref{Proposition Restriction Kinking Spin} follows from Identity (\ref{Spin Identity1}) that was obtained in the proof of Theorem \ref{Theorem General} and Theorem \ref{Theorem  Gravitational Kinks}.

\end{document}